\documentclass[a4paper,12pt]{amsart}
\usepackage{amsmath,amsfonts,amssymb,amsthm,amscd}
\usepackage{latexsym,graphicx}
\usepackage[matrix,arrow,ps,color,line,curve,frame]{xy}
\usepackage[usenames,dvipsnames]{color}

\usepackage{xcolor}

\renewcommand{\[}{\begin{equation}}
\renewcommand{\]}{\end{equation}}

\newtheorem{proposition}{\sc Proposition}[section]

\newtheorem{theorem}[proposition]{\sc Theorem}
\newtheorem{conjecture}[proposition]{\sc Conjecture}

\theoremstyle{definition}

\theoremstyle{remark}

\renewcommand{\phi}{\varphi}
\renewcommand{\epsilon}{\varepsilon}

\def\C{{\mathbb C}}
\def\N{{\mathbb N}}
\def\R{{\mathbb R}}

\def\z2{{\mathbb Z}_2}

\def\id{{\rm id}}

\def\half{{\mbox{$\frac{1}{2}$}}}

\def\half{{\mbox{$\frac{1}{2}$}}}

\usepackage[T1]{fontenc}
\usepackage{stmaryrd}

\begin{document}
\baselineskip=16pt

\author{\vspace*{5mm}Ludwik D\k abrowski}
\address{
SISSA (Scuola Internazionale Superiore di Studi Avanzati)\\
Via Bonomea 265, 34136 Trieste, Italy
}
\email{dabrow@sissa.it}
\title[Towards a noncommutative Brouwer fixed-point theorem]
{\large 
Towards a noncommutative\\
\vspace*{2.5mm}
Brouwer fixed-point theorem \\}

\begin{abstract}
\vspace*{5mm}
We present some results and conjectures on a generalization to the noncommutative setup 
of the Brouwer fixed-point theorem from the Borsuk-Ulam theorem perspective.
\end{abstract}
\vspace*{5mm}
\maketitle

\section{Introduction}
\noindent
The Borsuk-Ulam theorem \cite{b-k33}, a fundamental theorem of topology, is often formulated in one of three equivalent versions,
that regard continuous maps 
whose either domain, or codomain, are spheres. 
\begin{theorem}
\label{BUT}
For any $n\in\N$ let $\sigma: x\mapsto -x$ be the antipodal involution of $S^n$. Then:
\item{(0).}
If $F\colon S^n\to\mathbb{R}^n$
is continuous, then there exists 
$x\in S^n$, such that $F(x)=F(\sigma(x))$.
\item{(i).}
There is no
continuous  map ${f}\colon S^n\to S^{n-1}$, 
such that $ f\circ \sigma = \sigma\circ f$.
\item{(ii).}
There is no continuous map $g : B^n \to S^{n-1}$, such that $ g\circ \sigma (x) = \sigma\circ g (x)$ for all 
$x\in \partial B^n = S^{n-1}$.
\item{} The statements (0), (i) and (ii) are equivalent. $\diamond$
\end{theorem}
\noindent
Here $S^n$ is the unit sphere in $\R^{n+1}$ 
and $B^n$ is the unit ball in $\R^{n}$, but  homeomorphic spaces work as well.
The antipodal involution $\sigma: x\mapsto -x$
generates a free action of the group $\z2$ and we call
{\em $\z2$-equivariant} those
maps that commute with $\sigma$.

The Borsuk-Ulam theorem has a large variety of proofs, nowadays usually employing the {\em degree} of a map (the case 
$n = 1$ is easily seen by the intermediate value theorem). 
Here we explain the equivalence of (0) and (i) 
in Theorem \ref{BUT}.
Indeed, the logical negation of (0) would provide a map given by \[
{f}(x):=\frac{F(x)-F(-x)}{\|F(x)-F(-x)\|},
\]
contradicting (i).
Conversely, $f$ viewed as a map into $\R^n$ would provide a counterexample to  Theorem \ref{BUT}\,$(0)$.
Instead the proof of the equivalence of 
(i) and (ii) will be a special instance of the proof we shall give of the more general Proposition \ref{xBUT}.

The Borsuk-Ulam theorem has 
lot of applications to differential equations, combinatorics (e.g.  partitioning,
necklace division), Nash equilibria, and others, see e.g. \cite{m-j03}.

The well known equivalent theorems are 
the Lusternik-Schnirelmann theorem (that at least one 
among n + 1 open or closed sets covering $S^n$   contains a pair of antipodal points), 
and combinatorial Tucker's lemma and Fan's lemma 
\cite{nss}. 

There is also plentiful of corollaries.
Some of them are "fun facts", e.g. 
the case $n = 2$ is often illustrated by saying that at any moment there is always a pair of antipodal points on the Earth's surface with equal temperature and pressure.
But also, that squashing (with folding admitted) a balloon onto the floor there always is a pair of antipodal points one on the top of the other.

A famous corollary known as {\em Ham Sandwich theorem}
(sometimes named {\em Yolk, white and the shell of an egg} theorem) states that for any compact sets 
$V_1,\dots , V_n$ in $\R^n$ we can always find a hyperplane dividing each of them in two subsets of equal volume.
Another impressive implication is that no subset of $\R^n$ is homeomorphic to $S^n$. 
Also the famous Brouwer fixed-point theorem,
which can be formulated in two equivalent versions,
 is a corollary.
\begin{theorem}
\label{Brouwer}
\item{(I).}
A continuous map from the ball $B^n$ to itself has a fixed point.
\item{(II).}
There is no continuous map $g : B^n \to S^{n-1}$ that is the identity on the boundary $\partial B^n = S^{n-1}$.
\item{} The statements (I) and (II) are equivalent and 
are corollary of Theorem \ref{BUT}\,(ii). $\diamond$
\end{theorem}
\vspace*{-4mm}
\begin{proof}
\noindent
We show the equivalence of  (I) and (II), 
by showing the equivalence of their logical negations.\\
If there is $g : B^n \to S^{n-1}$ such that $g|_{S^{n-1}}=\id$, then,
$\sigma\circ g$ would have no fixed point,
which shows (I)$\Rightarrow$(II). \\
Next, assume there exists $h:B^n\to B^n$ such that $h(x)\neq x$, $\forall x\in B^n$. Then $g(x)$, defined as the intersection point of the half-line passing from $h(x)$ through $x$ with $\partial B^n = S^{n-1}$,
would define a continuous map $g : B^n \to S^{n-1}$ such that $g|_{S^{n-1}}=\id$, which shows (II)$\Rightarrow$(I). \\
 Finally note that $g$ as above is in particular 
$\z2$ equivarant on ${S^{n-1}}$, 
which shows (a contrary) that (II) (and thus also (I))
is a corollary of Theorem \ref{BUT}\,(ii).
\end{proof}

\vspace*{-1mm}
This theorem  can in turn be exemplified for 
$n=2$  by recognizing that, say on a map of Lazio,
placed on the table of the lecture hall at Villa Mondragone, there must be some point lying directly over the point that it represents.
Instead the case $n=3$ is usually illustrated by stirring a cup of coffee.

The Brouwer fixed-point theorem has lot of important applications too, and is also known to be 
equivalent to Knaster-Kuratowski-Mazurkiewicz lemma (for coverings), or to combinatorial Sperner's lemma
(with a following Fair Division result).

\section{Generalizations}

There have been numerous generalizations and strengthenings of the Borsuk-Ulam theorem,
see e.g. the comprehensive survey 
\cite{s-h85}, with almost 500 references before 1985.
For some more recent generalizations regarding the dimension of the coincidence set of the maps $f$ or $g$ for more general manifolds, or for homology spheres, and equivariance under other groups see e.g. 
\cite{t-ya07} and \cite{or-m}, and references therein.
Here we shall briefly present few generalizations,
whose noncommutative analogues could be most accessible, in our opinion.

\subsection{Going beyond spheres}

The version (0) of the Borsuk-Ulam Theorem  \ref{BUT}
employs a linear structure, and is not clear how it generalizes to spaces more general than spheres.
Instead the versions (i) and (ii) can indeed be generalized as follows.

For that view $S^{n}$ as the non reduced suspension $\Sigma S^{n-1}$ of $S^{n-1}$,
i.e. the quotient of $[0,1]\times S^{n-1}$
by the equivalence relation $R_\Sigma$ generated by
\begin{align}
&(0,x)\sim (0,x'),\qquad (1,x)\sim (1,x').
\end{align}
In this homeomorphic realization the $\z2$-action
becomes 
\[
(t,x)\mapsto (1-t,-x).
\]
Furthermore notice that the ball $B^n$ is homeomorphic to the cone $\Gamma S^{n-1}$ of $S^{n-1}$,
i.e. the quotient of $[0,\half]\times S^{n-1}$
by the equivalence relation $R_\Gamma$ generated by
\begin{align}
&(0,x)\sim (0,x').
\end{align}

\begin{proposition}\label{xBUT}
Let $X$ be  a compact space $X$ of finite covering dimension with a free $\z2$-action given by an involution $\sigma$.
\item (i).
There is no $\z2$-equivariant continuous  map ${f}\colon 
\Sigma X \to X $,
where the $\z2$-action on $\Sigma X$ is given by the involution
\[
(t,x)\mapsto (1-t,\sigma(x)),
\]
\item (ii). 
There is no continuous map $g : \Gamma X \to X$ that is $\z2$-equivariant 
on $\partial\, (\Gamma X) = X$.
\item
The conditions (i) and (ii) are equivalent.
$\diamond$
\end{proposition}
\begin{proof}
The statement (i) follows from an observation 
in \cite{y-m13} (Remark 17).
We show the equivalence of (the negations of) (i) and (ii).\\
Let $i: \Gamma X \to \Sigma X$
be the canonical injection.
Given $\z2$-equivariant $f:\Sigma X\to X$, 
$g:= f\circ i: \Gamma X \to X$  is 
$\z2$-equivariant on $\partial\, (\Gamma X) = X$.\\
Viceversa, given $g$ such that $g|_X$ is $\z2$-equivariant, $f$ defined by
$$
f(t,x)= 
\left\{ 
\begin{array}{ll}
g(t,x), &\mbox{if}\;\;  ~t\in[0,\half];\\
{}&{}\\
\sigma\big(g(1-t,\sigma(x)\big), &\mbox{if}\;\;  ~t\in[\half, 1],
\end{array}
\right.
$$ 
is continuous and $\z2$-equivariant.
\end{proof}
We remark that the statements (i) and (ii) are equivalent without the assumption that $X$ has finite covering dimension. 
It is interesting to know if it  
can be avoided in the Proposition \ref{xBUT} 
as a special case of the conjecture in \cite{BDH}, 
or rather some assumptions should be added to the latter one.

Concerning the Brouwer theorem, its usual version Theorem\ref{Brouwer}\,(I)  
has various important generalizations (in particular to certain topological linear spaces). 
We propose a conjecture about
extensions to more general spaces 
that may lack an underlying linear structure: 
\begin{conjecture}\label{xBrouwer}
If $X$ is a compact space $X$ of finite covering dimension with a free action of $\z2$,
then a continuos map from $\Gamma X$ to itself has a fixed point. $\diamond$
\end{conjecture}
It is not clear however how this conjecture could follow from Proposition \ref{xBUT}\,(ii), since the linear structure of $\R^{n+1}$ present in the proof for $S^n$ is missing. 
Instead, we have the following generalization of the retract version Theorem\ref{Brouwer}\,(II) of the Brouwer theorem, which is a direct corollary of Proposition \ref{xBUT}\,(ii):
\begin{proposition}
\label{xBrouwer2}
If $X$ is a compact space of finite covering dimension with a free action of $\z2$,
then there is no continuous map $g : \Gamma X \to X$ that is the
identity on the boundary $\partial (\Gamma X) = X$.
$\diamond$
\end{proposition}
\noindent
This proposition is a weaker statement than Conjecture \ref{xBrouwer}. Indeed if there was a map $g$ as above,
then $\sigma\circ g$ would not have fixed points.
It is interesting to know if the implication holds also the other way, or to find a counterexample.

\subsection{Going beyond  $\z2$.}

Among generalizations to other groups $G$, whose 
noncommutative analogues in our opinion might be most accessible, are certain necessary homological conditions for the existence of a $G$-equivariant map $f:S(W) \to S(V)$ of spheres in the vector spaces $W, V$  of orthogonal representations
of a compact Lie group $G$ \cite{m-w94} and of 
Hausdorff, pathwise connected and paracompact  free $G$-spaces \cite{bm06}.

We should mention that in \cite{BDH} a conjecture 
regarding a wide class of 
generalizations of the Borsuk-Ulam theorem 
was proposed in terms of \emph{joins}
of a compact Hausdorff space $X$ equipped with a continuous free action of a non-trivial compact Hausdorff group~$G$. 
Particular cases of that conjecture going beyond the Borsuk-Ulam Theorem have already been studied in \cite{s-e40,g-h37,g-h43}
(see also~\cite{t-a12} for weaker results for non-free  
$\mathbb{Z}_p$-actions and maps from $X$ to~$S^1$), and moreover it holds true for the case $X=G$.

We close this section by a remark that finding other examples and conjectures on pairs of $G$-spaces $X$, $Y$ that admit a  $G$-equivariant  injective map $X\to Y$ but no $G$-equivariant maps $Y\to X$,
can be helpful for study of both the spaces $X$ and $Y$, and this applies to noncommutative spaces too.

\section{Noncommutative framework}

\subsection{Borsuk-Ulam theorem for quantum spheres}

The idea is to apply the Gelfand transform
and then to substitute commutative algebras of functions by suitable noncommutative algebras. 
But 'quantization may remove degeneration' and 
the three versions of Theorem~\ref{BUT} may behave
differently in this process. 
The usual version (0) formulated in terms of points is problematic, so we focus on the versions (i) and (ii).
Applying the Gelfand transform, Theorem~\ref{BUT}\,(i) translates to:
\[\label{classphere}
\text{There is no}
\;\;\text{$\z2$-equivariant $*$-homomorphism}\; 
\phi:C(S^n)\longrightarrow C(S^{n+1}) \, .
\]
The next step will be to investigate for which 
noncommutative C*-algebras $A$ and $B$ of 
suitable quantum spheres, \eqref{classphere}
is satisfied with $C(S^{n})$ and $C(S^{n+1})$ substituted by $A$ and $B$. 
It holds for a family of $q$-deformed quantum spheres $S_q^n$ introduced 
in~\cite{p-p87,vs90,hs03}, (see also \cite{hl04} for isomorphic euclidean quantum spheres). 
This is shown in \cite{bhr} (see also \cite{BDH})
for $A=C(S_q^1)=C(S^1)$ and $B=S_q^2$,
(the equatorial Podle\'s quantum two-sphere),
while it is proven in~\cite[Theorem 3]{y-m13} for $A=C(S_q^n)$ and $B=C(S_q^{n+1})$ with arbitrary $n$.

A quantum analogue of \eqref{classphere} also holds 
\cite{p-b} 
for the family of $\theta$-deformed speres $A=C(S_\theta^n)$ and $B=C(S_\theta^{n+1})$ with arbitrary $n$, where $S_\theta^{n}$ for odd $n$ have been defined in \cite{no97} ($n=3$ is a special case 
of \cite{k-m91}), and for even $n$ in \cite{cl01}.

These results follow from a detailed study 
of a stringent invariance under the deformation quantization in the equivariant KK-theory (respectively K-theory). 
Actually it should be interesting to study also the semiclassical limit of these deformations,
maps between them, and the behaviour of the relevant Poisson structures.

We expect that \eqref{classphere} works also for other families of quantum spheres. 
But both for $S_q^n$ and for $S_\theta^{n}$, $n$ odd,  
it is not clear in general how to formulate the quantum analogue of Theorem  \ref{BUT}\,(ii) and of 
Theorem \ref{Brouwer}\,(II)
without introducing first corresponding quantum balls. 
Therefore I will present an alternative proposal.

\subsection{Borsuk-Ulam and Brouwer conjectures for C*-algebras with free $\z2$-actions}

For any $C^*$-algebra $A$ let 
\begin{equation}
\Sigma A:= 
\left\{\alpha\in C([0,1], A) \, \Big|\, \mathrm{ev}_0 (\alpha) \in \C,\, \mathrm{ev}_1 (\alpha) \in  \C \right\},
\end{equation}
where $\mathrm{ev}_t$ is the evaluation at $t$,
be its suspension $C^*$-algebra.
Obviously, $C([0,1], A)$ is isomorphic to $C([0,1]) \otimes A$ and $\Sigma C(X)$ is isomorphic to $C(\Sigma X)$
for any compact topological space $X$.
An involution $\sigma$ on $A$ induces the involution $\sigma_\Sigma$ of  $\Sigma A$, given by 
$$\big(\sigma_\Sigma (\alpha)\big)(t) = 
\sigma \big( \alpha(1-t)\big).$$
%
If the $\z2$ action generated by $\sigma$ is free,
then the $\z2$ action generated by $\sigma_\Sigma$  
is also free.

\begin{conjecture}\label{conjqsphere}
For a unital $C^*$-algebra $A$ with a free action of 
$\z2$,
\[\label{qsphere}
\text{there is no}
\;\;\text{$\z2$-equivariant 
$*$-homomorphism}\; \phi:A\longrightarrow \Sigma A.
\quad\diamond
\]
\end{conjecture}

Note that the special case when $A=S_\theta^{n}$, $n$ even,
make this conjecture true, but $A=S_\theta^{n}$, $n$ odd
or $A=C(S_q^n)$ (unless the relevant deformation parameter is $\theta=0$ or $q=1$), is different from the results of \cite{y-m13} or \cite{p-b} - and thus not covered by them, since then the suspension of the quantum $n$-sphere is not isomorphic to the quantum 
$(n+1)$-sphere.
 Perhaps however a more "noncommutative" notion of suspension could be introduced (with a noncentral parameter $t$), which would encompass these families as such iterated suspensions of $S^1$
 (see e.g. \cite{hs08} and references therein).

Next we discuss the dualization of Theorem~\ref{BUT}\,(ii).
Using the homeomorphism $B^n=\Gamma S^{n-1}$,
it translates via the Gelfand transform to:
\[\label{classball}
\text{There is no}
\;\;\text{$\z2$-equivariant $*$-homomorphism}\; 
C(S^{n-1})\longrightarrow C(\Gamma S^{n-1}) \, .
\]
Replacing the commutative C*-algebras of functions on the sphere $S^{n-1}$ by noncommutative C*-algebra $A$ 
of a quantum sphere, and properly dualizing the maps,
we obtain a noncommutative version of \eqref{classball}:
\begin{conjecture}\label{conjqball}
For a unital $C^*$-algebra $A$ with a free action of 
$\z2$,
\[\label{qball}
\text{there is no $\z2$-equivariant
$*$-homomorphism\, $\gamma:A\longrightarrow \Gamma A$.}
\quad\diamond
\]
\end{conjecture}
Here 
\[
\Gamma A := 
\left\{\alpha\in C([0,\half ],A) \, \Big|\, \mathrm{ev}_0 (\alpha) \in \C \right\}
\]
is the cone of $A$ and the $\z2$-equivariance means
\[
\sigma\circ \mathrm{ev}_\frac{1}{2}  \circ\gamma
=\mathrm{ev}_\frac{1}{2}  \circ \gamma\circ\sigma,
\]
where the involution $\sigma$ generates the action 
of $\z2$ on $A$. 
Notice that this is in perfect agreement with the commutative situation since there are canonical isomorphisms $C([0,\half ],A)=C([0, \half ]) \otimes A$ and  $C(\Gamma S^{n-1})=\Gamma (C(S^{n-1}))$,
which are equivariant with respect of the the $\z2$-action on the "boundaries" $A$ and $S^{n-1}$, respectively.

\begin{proposition}\label{propA}
For any unital $C^*$-algebra $A$ with a free action of 
$\z2$, there exists
a $\z2$-equivariant $*$-homomorphism\, $\phi:A\longrightarrow \Sigma A$
if and only if there exists a $\z2$-equivariant
$*$-homomorphism\, $\gamma:A\longrightarrow \Gamma A$.
Thus the  conjectures \ref{conjqsphere} and \ref{conjqball} are equivalent. $\diamond$
\end{proposition}
\begin{proof}
Let 
$\iota : \Sigma A \to \Gamma A$ be the restriction to $[0,\half]$.
If $\phi:A\longrightarrow \Sigma A$ is a $\z2$-equivariant  *-homomorphism, then
$\gamma:= \iota\circ \phi, \gamma:A \to \Gamma A$ is 
a $\z2$-equivariant *-homomorphism.\\
If  $\gamma:A \to \Gamma A$ is a $\z2$-equivariant *-homomorphism, then $\phi:A \to \Sigma A$ defined for any $a\in A$ by 
$$\label{acr}
\big(\phi(a)\big)(t)= \left\{ 
\begin{array}{ll}
\big(\gamma(a)\big)(t), &\mbox{if}\;  ~t\in[0,\half];\\
{}&{}\\
\Big(\sigma_\Sigma\circ\gamma\circ\sigma(a)\Big)(t), &\mbox{if}\;  ~t\in[\half, 1],
\end{array}
\right.
$$
is $\z2$-equivariant.
\end{proof}

Concerning the Brouwer fixed-point theorem,
it is not clear not only how to extend its usual version Theorem \ref{Brouwer}\,(I) 
to a more general class of spaces so that 
it would be a corollary of \eqref{qball}, but 
also how to formulate it for a noncommutative unital C*-algebra,  due to the presence 
of a {\em point} (singleton) in its statement.
However, we can offer the following conjecture
for a noncommutative generalization of the retract version (II) of the Brouwer theorem Theorem\ref{Brouwer}, which is an immediate corollary of Conjecture \ref{conjqball}:
\begin{conjecture}\label{conjborsuk}
For a unital $C^*$-algebra $A$ with a free action of $\z2$ there is no 
{$*$-homomorphism}\, $\gamma:A\longrightarrow \Gamma A$, such that $\mathrm{ev}_\frac{1}{2}  \circ \gamma$ is the identity on $A$. $\diamond$
\end{conjecture}

We mention that the $n$-dimensional quantum $\theta$-balls have been defined for $n$ even in \cite{p-ma13} and they are just isomorphic to $\Gamma S_\theta^{n-1}$. 
Thus Conjecture \ref{conjqball}, Proposition \ref{propA} and Conjecture \ref{conjborsuk} hold true for these families.

\subsection{Borsuk-Ulam conjectures for C*-algebras with free quantum group actions}

We close this note by mentioning that two conjectures about a much wider noncommutative generalization of \eqref{classphere} have been proposed in \cite{BDH}. 
They are formulated within the framework of actions of compact quantum groups on unital
\mbox{C*-algebras} and use the noncommutative analogue 
of the equivariant join as defined in \cite{dhh}.

The first of the conjectures states that, if $H$ is the C*-algebra of a compact quantum group coacting freely on a unital C*-algebra $A$,  
then there is no equivariant $*$-homomorphism from $A$ to the join C*-algebra $A$ with $H$.  
This encompasses as a special case the conjecture \ref{conjqsphere} for $A=C(S_q^n)$, but still does not include the result shown in \cite{y-m13} for the quantum spheres  $A=C(S_q^n)$.
For $A$ being the C*-algebra
of continuous functions on a sphere with the antipodal coaction of 
the C*-algebra of functions on $\mathbb{Z}/2\mathbb{Z}$,
the Theorem \ref{BUT}\,(i) is recovered.

The second conjecture states that there is no equivariant $*$-homomorphism 
from $H$ to the join C*-algebra of $A$ with $H$.  
In \cite{BDH} it is shown how to prove the conjecture in the special case $A=C(SU_q(2))=H$, which
is tantamount to showing the non-trivializability of the quantum instanton principal $SU_q(2)$-bundle \cite{p-m94}.

It is interesting to formulate a 
corresponding noncommutative Brouwer conjectures, 
in particular in such a way that they would follow as corollary.

\section*{Acknowledgements}
\noindent
It is a pleasure to thank P.M.~Hajac for discussions and useful comments. This work was partially supported by PRIN 2010-11 grant `Operator Algebras, Noncommutative Geometry and Applications' and  by grant 2012/06/M/ST1/00169\, `HARMONIA`.

\vspace*{1.25mm}

\end{document}